\documentclass[11pt]{article}

\usepackage{amsmath}
\usepackage{amsthm,amscd,amsfonts}
\usepackage{amssymb, upref, color}
\usepackage{amsmath,amssymb,amsthm,mathrsfs}
\usepackage{color}
\usepackage[dvips]{graphicx}
\usepackage{epsf,epsfig, verbatim} 
\usepackage{latexsym,bm}
\usepackage{indentfirst}
\usepackage{color}
\usepackage{xcolor}
\usepackage{cite}

\usepackage{epsfig}
\usepackage{epstopdf}
\usepackage{subfigure}
\usepackage[font=scriptsize]{caption}
\usepackage{graphicx}
\graphicspath{{figures/}}

\numberwithin{equation}{section}

\theoremstyle{plain}
\newtheorem{exam}{Example}[section]
\newtheorem{theorem}[exam]{Theorem}
\newtheorem{lemma}[exam]{Lemma}
\newtheorem{remark}[exam]{Remark}

\def\bc{\begin{center}}
 \def\ec{\end{center}}
\def\be{\begin{equation}}
 \def\ee{\end{equation}}
\def\ba{\begin{array}}
 \def\ea{\end{array}}
\def\bea{\begin{eqnarray}}
 \def\eea{\end{eqnarray}}
\def\beaa{\begin{eqnarray*}}
 \def\eeaa{\end{eqnarray*}}

\def\la{\lambda}

\def\e{\varepsilon}

\def\Lp{{\mathcal L}^p}

\def\lb{\label}

\begin{document}
 \sloppy
 \captionsetup[figure]{labelfont={bf},name={Fig.},labelsep=period}
 \captionsetup[table]{labelfont={bf},name={Table},labelsep=period}

 \title{\bf A dynamical system framework yielding quantitative inverse spectral results for Sturm–Liouville operators
\footnote{This paper was supported from Natural Science Foundation of Guangdong Province (2026A1515011073),
 		National Natural Science Foundation of China (No.12571165) and NSF-Zhejiang Province (No.LZ24A010006).}
 }

 \author
 {
  Yuchao He$^{a}$,\qquad 
  Yonghui Xia$^{b}$\footnote{Corresponding author.  yhxia@zjnu.cn;xiadoc@163.com}, \qquad Meirong Zhang$^{c}$\\
  {\small \textit{$^a$ School of Mathematical  Science,  Zhejiang Normal University, Jinhua 321004, China}}\\
  {\small \textit{$^b$ School of Mathematics, Foshan University, Foshan 528000, China }}\\
  {\small \textit{$^c$Department of Mathematical Sciences, Tsinghua University, Beijing 100084, China}}\\
  {\small Email:  yuchaohe@zjnu.edu.cn;  xiadoc@163.com; zhangmr@tsinghua.edu.cn}
 }

 \maketitle
 \begin{abstract}
         This paper establishes a dynamical-system framework that yields quantitative results for the inverse optimal spectral problem of reconstructing a potential $\hat{q}$ from finite observed eigenvalues to achieve an optimal approximation of the target potential $q_0$.
          Previous efforts relying on convex analysis have been confined solely to {\em qualitative} analysis due to the inherent limitations of convex-analytic techniques for inverse problems, while the {\bf quantitative} counterpart has remained an open problem. Based on our dynamical-system framework, we provide a quantitative characterization of the relationship between the reconstructed potential $\hat{q}$, its target potential  $q_0$, and the observed eigenvalue $\lambda_*$.
          In particular, for ${q} \in \mathcal{L}^2$, our framework yields a substantially stronger conclusion.
     Remarkably, our dynamical-system framework secures the uniqueness of  $\hat{q}$ over the full parameter space $(\lambda_*, q_0)$, liberating the theory from the prevailing constraint  $\lambda_* > \lambda_1(q_0)$ (where $\lambda_*$ is the observed eigenvalue and $\lambda_1$ is the principle eigenvalue). This stands in sharp contrast to classical approaches, which rely heavily on convex-set analysis and are inherently confined by its stringent assumptions.    
          An additional finding is the construction of a homeomorphic mapping that reveals the dilation relation between the errors $\|\hat{q} - q_0\|_{\mathcal L^p}$ associated with the $m$-th eigenvalue and the principal eigenvalue.    
     A summary of the main results, along with practical applications in structural health monitoring and damage detection, material design, seismic wave analysis, sonar detection, and related fields, concludes this work.
     \\
 {\footnotesize{\bf Keywords}: Inverse spectral problem; Planar dynamical system approach; Potential reconstruction; Sturm-Liouville equation}
\end{abstract}
 \section{Introduction and motivation}

  Since the pioneering work of Sturm \cite{Sturm} and Liouville \cite{Liouville}, numerous mathematicians have made significant contributions to the fundamental theory of the Sturm-Liouville equation. Notable examples include the contributions of Arnold \cite{arnold-book, arnold}, Courant and Hilbert \cite{C-H}, Bérard and Helffer \cite{P-B}, Bérard and Webb \cite{P-D}, Krein \cite{krein1}, Ashbaugh and Benguria \cite{M-R,M-R2}, M\"oller and Zettl \cite{M-Z}, and Zettl \cite{zettl}. 
 Note that the standard form of the Sturm-Liouville equation with Dirichlet boundary conditions is given as: 
\begin{equation}\label{s-l}
	\begin{cases}
		-u''+q(x)u=\la u,\\
		u(0)=u(1)=0,
	\end{cases}
\end{equation}
where the potential {$q\in \mathcal L^p:=\mathcal L^p([0,1],\mathbb R)$, $p\in[1,+\infty]$.}
According to the classical spectral theory (e.g. see \cite{book1, book2}), \eqref{s-l} admits an increasing sequence of eigenvalues
\[
\lambda_1(q)<\lambda_2(q)\cdots<\lambda_m(q)<\cdots.
\]
 The origins of the classical direct spectral problem can be traced back to the work of Krein \cite{krein1} in the 1950s. 
  The task of the direct problem for Sturm-Liouville operator is: for a prescribed potential $q(x)$ and specified boundary conditions, one seeks to find the eigenvalues $\lambda_m$, the associated eigenfunctions $E_m(x)$, and and an understanding of their analytical properties. By index theory, Hu \cite{Hu1} and Hu et al \cite{Hu2,Hu3} studied the eigenvalues of Sturm-Liouville problem.
 We recall that Zhang \cite{zhang1} proved that the eigenvalues exhibit complete continuity with respect to varys in the potential, which established a foundation for extensive follow-up research concerning direct problems in Sturm-Liouville equations. 
 Subsequently, Wei, Meng and Zhang \cite{zhang2}  determined the maximum and minimum eigenvalues via a given potential function $q\in \mathcal L^p$, $\ p\in[1,+\infty]$. {It is noteworthy that the methodology developed by Zhang \cite{zhang1} for the Sturm-Liouville operators extends its significance to the study of other classes of operator problems (see Chu, Meng and Zhang \cite{c-m-z, cmz_ADV} and Zhang \cite{Zhang3}).} Very recently, based on the results of Zhang \cite{zhang1}, Chu, Meng and Zhang \cite{chu2} developed an analytical approach to derive sharp estimates for the lowest positive periodic eigenvalue and all Dirichlet eigenvalues of a general Sturm-Liouville problem.
 
 In contrast to the direct spectral problem, the inverse spectral problem involves recovering the unknown potential function $q(x)$ or the boundary conditions from given spectral data. This field was initiated by the seminal work of Borg \cite{Borg}, who proved that a single spectrum is generally insufficient to uniquely determine the potential. A more complete characterization was later established by Levitan and Gasymov \cite{LG} and Gelfand and Levitan \cite{GL}, who developed the classical reconstruction method using spectral measures. The field's scope was further broadened by Kac \cite{kac}, whose famous inquiry, ``Can one hear the shape of a drum?'', stimulated extensive research into the relationship between eigenvalues and domain geometry. Building on these foundations, P\"oschel and Trubowitz \cite{poschel} advanced the theory by establishing the real analytic isomorphism of the spectral mapping and describing isospectral sets as infinite-dimensional tori. 
  {However, traditional inverse problem methods \cite{hald, AMIROV,E-K} are difficult to apply directly; while the traditional inverse spectral problem typically requires the complete spectrum, the associated inverse nodal problem \cite{P-T,YANGCF,he,yang,yang2,w-y,shen,C-J,C-T, P-B-R,tams} depends on a dense subset of nodes (zeros of the eigenfunctions). In many real-world settings-including structural health monitoring, damage detection, material design, seismic wave analysis, and sonar detection—the observed data (e.g., frequencies, nodal information) are often sparse and incomplete.}
  
   Can one detect the structural damage in bridge stay cables or architectural beams? Can one correct a string's off-pitch tone by reshaping? Can one detect seismic waves by listening and using prior geological knowledge?  Can one design an advanced composite material by precise resonant frequencies and an initial density distribution? To address these practical questions, this paper develops a quantitative theoretical framework that reveals the profound mathematical logic inherent in them.  We intend to delve directly into the core mathematical principles, reserving the detailed answers to the aforementioned questions for the final section, where they will be presented as practical applications of our mathematical theory. 
  In essence, all the questions above can be unified as an {inverse optimization spectral problem} for Sturm-Liouville operator: reconstruct a potential (physical state) $\hat{q}$ from limited spectral data (the observed resonant frequencies), where $\hat{q}$ optimally approximates the target potential $q_0$. In rigorous mathematical terms, the inverse optimization spectral problem (IOSP) described above can be formulated as follows:
  \\
  {\bf Problem ${ \rm (IOSP)}$}: Let $\lambda_* \in \mathbb R$ be the given eigenvalue and $q_0 \in \mathcal L^p$ be the prior potential. The problem seeks a potential $\hat q \in \mathcal L^p$ that minimizes the deviation from $q_0$ while satisfying the constraint $\lambda_m(\hat{q}) = \lambda_*$:
  \begin{equation}\label{youhua}
  	\|\hat q-q_0\|_{\mathcal L^p}=\min\{\|q-q_0\|_{\mathcal L^p}:\lambda_*=\lambda_m(q),\ q\in \mathcal L^p\}.
  \end{equation}
       Compared with traditional inverse spectral problems, Problem (IOSP) exhibits superior applicability in real-world scenarios due to its minimal data requirements, necessitating only a single observation.  
 Note that a qualitative result was established in \cite{V-I} for the existence and uniqueness of the reconstructed potential $\hat q$ by the restriction 
 $\lambda_*>\lambda_1(q_0)$. 

 However,  the reliance on convex analysis theory in \cite{V-I} to discuss the uniqueness of the inverse spectral problem restricts their result to the traditional constraint $\lambda_* > \lambda_1(q_0)$.   
 Furthermore, in practical applications, proving the {\em qualitative} result for existence and uniqueness of the solution to Problem (IOSP) is necessary but insufficient. The essential goal is a {\bf quantitative} characterization of the reconstructed potential $\hat{q}$, for which an explicit exact expression represents the optimal outcome.
    We introduce a planar dynamical system approach that fundamentally departs from all previous methodologies in inverse spectral problems. This approach simultaneously eliminates the restrictions on the parameter space $(\lambda_*, q_0)$ and yields an explicit expression for $\hat{q}$.
    On one hand, in contrast to the uniform dynamics observed under the restriction $\lambda_* > \lambda_1(q_0)$,
  removing this constraint exhibits three distinct types of dynamical behavior, 
  significantly increasing analytical difficulty, as evidenced by the phase diagrams {(see Fig. 1, Fig. 2., Fig. 3 in Section 3).  
 	On the other hand, we provide specific expressions for the minimized error  $\|\hat{q} - q_0\|_{\mathcal L^p}$, which is primarily derived from the periodicity of the solution to the critical equation. Moreover, in the specific case where ${q} \in \mathcal{L}^2$, we demonstrate that a sharper result can be achieved: an explicit analytical expression for $\hat{q}$ in terms of $q_0$ and $\lambda_*$. 
 	 	Furthermore, we present the dilation relation between the errors  associated with the $m$-th eigenvalue and that of the principal eigenvalue through  a constructed homeomorphic mapping. This homeomorphism shows our dynamical-system framework accommodates any \(\lambda_m(q)\) as \(\lambda_*\) in Problem (ISOP), rather than only \(\lambda_1(q)\).

In next section, we prove the existence of solutions to Problem (IOSP).
In Section 3, we establish a planar dynamical system framework to fully classify the critical equation. Specifically, we show that its solutions are periodic and determine the sign parameter
$\e$ explicitly in terms of $q_0$, $\lambda_*$ and $m$ distinguishing three distinct dynamical regimes (see Theorem \ref{ep_sign} and \ref{ep_sign2}). Based on this classification, we derive the explicit quantitative expression for the error $\|\hat{q} - q_0\|_{\mathcal{L}^p}$ (See Theorem \ref{Th6}), which  characterizes the relationship between $q_0$, $\lambda_*$, $m$ and the potential $\hat{q}$. We further sharpen the results for the special case $q\in\mathcal{L}^2$, obtaining an explicit analytical expression for $\hat{q}$ in terms of Jacobi elliptic functions (See Theorem \ref{3.10}). We then prove the uniqueness of $\hat q$ over the full parameter space $( \lambda_*, q_0)$ without the traditional constraint $\lambda_*>\lambda_1(q_0)$ (See Theorem \ref{3.12}). In Section 4, we introduce a homeomorphic mapping $\mathcal{R}_m$ and establish a dilation relation between the errors associated with the $m$-th eigenvalue and the principal eigenvalue. Finally, in Section 5, we summarize the contributions  and present practical applications to answer the questions arsing from the practical applications.

\section {Preliminary results: existence and critical equation}
\subsection{Existence of the solution to Problem (IOSP)}
We first demonstrate that for this inverse spectral problem, its solution exists, primarily relying on the complete continuity of eigenvalues with respect to the potential. 
 \begin{theorem}
 For $q_0 \in \mathcal  L^p$ $(p>1)$ , there exists a potential $\hat q \in \mathcal  L^p$ which
 solves Problem ${\rm (IOSP)}$.
 \end{theorem}

 \begin{proof}
  Let
  $$
  S_{\lambda_*}=\{q\in \mathcal L^p|\lambda_m(q)=\lambda_*\}.
  $$
  It follows from  $\partial_q \lambda_{m}(q)\ne 0$ that $S_{\lambda_*}$ is a well-defined   1-codimension differentiable sub-manifold of $(\Lp, \|\cdot\|_{\mathcal L^p})$. For problem ${\rm (IOSP)}$, there exists a minimizing potential sequence $\{q_n\}$ such that $q_n\in S_{\lambda_*}$ $(n\in \mathbb N)$ and
  \begin{equation}\lb{1.5}
   \|q_n-q_0\|_{\mathcal L^p}\to \inf_{q\in S_{\lambda_*}}\|q-q_0\|_{\mathcal L^p}<+\infty.
  \end{equation}
  It follows that  $\|q_n\|_{\mathcal L^p}$ $(n\in \mathbb N)$ are bounded. We may assume that  $q_n\rightharpoonup \hat q$. Moreover, we  obtain that 
  \begin{equation}\lb{1.6}
   \|\hat q-q_0\|_{\mathcal L^p}\leq\liminf\limits_{n\to\infty}\|q_n-q_0\|_{\mathcal L^p}=\inf_{q\in S_{\lambda_*}}\|q-q_0\|_{\mathcal L^p}.
  \end{equation}

  Since the eigenvalues are completely continuous in the potential \cite[Theorem 5.1]{zhang1}, we have
  \[
  \la_{m}(\hat q)=\lim_{n\to\infty}\la_{m}(q_n)=\lambda_*,
  \]
  i.e. $\hat q\in S_{\lambda_*}.$ It follows from \eqref{1.6} that
  \be \lb{RT}
  \|\hat q-q_0\|_{\mathcal L^p}=\inf_{q\in S_{\lambda_*}}\|q-q_0\|_{\mathcal L^p}=\min_{q\in S_{\lambda_*}}\|q-q_0\|_{\mathcal L^p}.
  \ee
  That is, $\hat{q}$ is the solution to the problem ${\rm (IOSP)}$.
 \end{proof}
 
 \subsection{The critical equation for the optimal potential}

Next,  we assume $\lambda_m(q_0) \neq \lambda_*$. Indeed, if $\lambda_m(q_0) = \lambda_*$, the problem {\rm (IOSP)} would collapse into the trivial case, where $\hat q = q_0$.

 In this subsection, we recast the  Problem ${\rm (IOSP)}$ into a solvable class of Schr\"odinger equations via the Lagrange multiplier method. This transformation facilitates a systematic investigation of the problem through conventional analytical frameworks.

 Problem ${\rm (IOSP)}$ is a minimization problem on the aim functional $\|\hat q-q_0\|_{\mathcal  L^p}$ with a constraint
 \[
 \lambda_m=\lambda_*.
 \]
 For an exponent $p \in (1, \infty)$, the increasing homeomorphism $\phi_p(s) : \mathbb R \to \mathbb R$ is defined by
\begin{equation}\lb{phi}
\phi_p(s)=|s|^{p-2}s,\quad {\rm for}\  s\in \mathbb R 
\end{equation}
and 
\[
\phi_p^{-1}(t)=|t|^{p^*-2}t,\quad {\rm for}\  t\in \mathbb{R},
\]
where $p^*=\frac{p}{p-1}
\in (1, \infty)$ is the conjugate exponent of $p$.\\
Note that the
Fr\'echet derivatives $\partial_q \la_m = E_m^2$
and
\[
\partial_q\|\hat q-q_0\|_{\mathcal  L^p}=\|\hat q-q_0\|_{\mathcal  L^p}^{1-p}\phi_p(\hat q-q_0)
\]
It follows from the Lagrange multiplier method that the minimizing potential $\hat q$ satisfies 
\begin{equation}\lb{L}
\|\hat q-q_0\|_{\mathcal  L^p}^{1-p}\phi_p(\hat q-q_0)=cE_m^2,\quad {\rm for} \ x\in[0,1],
\end{equation}
where the constant $c = c(\hat q) \ne 0$.\\
According to \eqref{phi}, \eqref{L} is transformed into 
\begin{equation}\lb{3.4}
\phi_p(\hat q(x)-q_0(x))=\widetilde{c}E_m^2(x),
\end{equation}
where $\widetilde{c}=\frac{c}{\|\hat q-q_0\|_{\mathcal  L^p}^{1-p}}$.\\
By \eqref{3.4}, let us introduce the function
\begin{equation}\lb{um}
 u_m(x)=\sqrt{|\widetilde{c} |}E_m(x).
\end{equation}
Then we can write \eqref{3.4} as
\begin{equation}\lb{hat q1}
|\hat{q}-q_0|^{p-2}(\hat q-q_0)=\e u_m^2,
\end{equation}
where $\e={\rm sign}(\widetilde{c})$.\\
Moreover, we obtain that
\begin{equation}\lb{guanxi}
\hat q-q_0=\e\phi_p^{-1}(u_m^2)=\e |u_m|^{2p^*-2}.
\end{equation}

 According to Equation \eqref{guanxi}, the solution $\hat{q}$ of problem ${\rm (IOSP)}$ is determined by $u_m$, where $u_m$ is the solution to the critical equation in the following theorem.
 
 \begin{theorem} {{\bf(Critical equation)}}
 	\label{Th2.2}
  For problem { \rm (IOSP)}, the  ‘eigenfunction’ $u_m(x)$ satisfies
 \begin{equation}\lb{cri}
  \begin{cases}
  -u_m''+q_0u_m+\e \phi_{2p^*}(u_m)=\lambda_* u_m,\quad 0\leq x\leq1,\\
  u_m(0)=u_m(1)=0.
  \end{cases}
 \end{equation}
 \end{theorem}
 \begin{proof}
 The initial conditions $u_m(0)=u_m(1)=0$ are
 evident.
 For the deduction of \eqref{cri},  let us notice from \eqref{um} and the equation for eigenfunction $u_m(x; \hat q)$ that
 \[
  -u_m''+\hat q u_m=\lambda_* u_m,\quad 0\leq x\leq1.
 \]
 By using \eqref{hat q1}, this equation is just \eqref{cri}.
  \end{proof}
 
 In fact, when $q_0$ is a constant, the value of $\e$ is determined by $\lambda_*$ and $q_0$, which will be explained in the next section.


{ \section{Construction of a dynamical-system framework for Problem (ISOP)}

In this section, we construct a dynamical-system framework that yields quantitative results for Problem (ISOP). For clear exposition, the section is split into four subsections. First, we employ planar dynamical system theory via phase-plane portraits to fully classify the critical equation \eqref{cri} derived from Theorem \ref{Th2.2}. Second, we derive the explicit quantitative relation linking $\lambda_*$, $q_0$ and $\hat q$. Third, our dynamical-system framework delivers sharper bounds for potentials $ {q} \in \mathcal{L}^2$. 
 Finally, we prove the uniqueness of the optimal potential $\hat{q}$ on the full parameter space $(\lambda_*, q_0)$.
}

\subsection{Classification of critical equation via planar dynamical system approach}
In what follows, we focus on the case where $q_0$ is a constant, which corresponds to the string restoration problem with an approximately homogeneous material or the structural health monitoring for the bridge components with a nearly uniform initial stiffness (detailed practical background is referred  to Subsection 7.2). This yields the first integral of \eqref{cri}:
 \begin{equation}\lb{shouci1}
	(u_m'(x))^2-\frac{\e}{p^*}|u_m(x)|^{2p^*}+(\lambda_*-q_0) u_m^2(x)= k,
\end{equation}
 where $k=(u_m'(0))^2$.
 
 To begin with, we demonstrate that the solutions to the critical equation exhibit periodicity.


 \begin{lemma}\label{zhouqi}
 	 For the parameters $\lambda_*, q_0\in \mathbb R$, the solution $u_m$ of \eqref{cri}   has a minimal period of $\frac{2}{m}$ and satisfies
 	 \begin{equation}\label{anti}
 	 u_m(x+\frac{1}{m})=-u_m(x),\quad x\in \mathbb R.
 	 \end{equation}
 \end{lemma}
 \begin{proof}
 Set \[A=\{x|x>0,\ u_m(x)=0\}.\] Since $u_m(1)=0$,  $A$ is not empty. Let $x_1\in A$ be  the minimum element of $A$. It follows from \eqref{shouci1} that $(u_m'(x_1))^2=(u_m'(x))^2=k$ and $|u_m'(x_1)|=|u_m'(0)|$.
 Based on $x_1$ being the first positive zero, we conclude that $u_m'(x_1)=-u_m'(0)$. Denote $\bar u_m(x):=-u_m(x+x_1)$.\\ 	Since 
 \[
 (\bar u_m(0),\bar u_m'(0))=(-u_m(0),-u_m'(0))=(0,u_m'(0))=(u_m(0),u_m'(0)),
 \]
 we can see that $\bar u_m(x)$ is  a solution of \eqref{cri}. 
 Based on the uniqueness of the solution to the first equation of \eqref{cri} under the given initial conditions, it follows that 
\[
u_m(x+x_1)=-u_m(t).
\] 
 Moreover, the positive zeros of $u_m(x)$ must be
\[
x_1<2x_1<\cdots<nx_1<\cdots.
\] 
Note that $u_m(x)$ has $m - 1$ zeros in $(0, 1)$. Consequently, we obtain that $x_1=\frac{1}{m}$. 
By satisfying \eqref{anti}, $u_m(x)$ exhibits anti-periodicity with a minimal period of $\frac{1}{m}$, which naturally implies a minimal periodicity of $\frac{2}{m}$.
 \end{proof}
 
 Based on the periodicity, we now prove the following theorem and clarify that when $\e=1$ or $-1$, the critical equation yields the two phase diagrams shown in Fig.~\ref{fig:img1} and Fig.~\ref{fig:img2}, respectively.
We will further explain the relationship between the value of $\e$ and the parameters $q_0$ and $\lambda_*$

 \begin{figure}[htbp]
  \centering
   \begin{minipage}{0.48\textwidth}
  	\includegraphics[width=\textwidth,height=6cm]{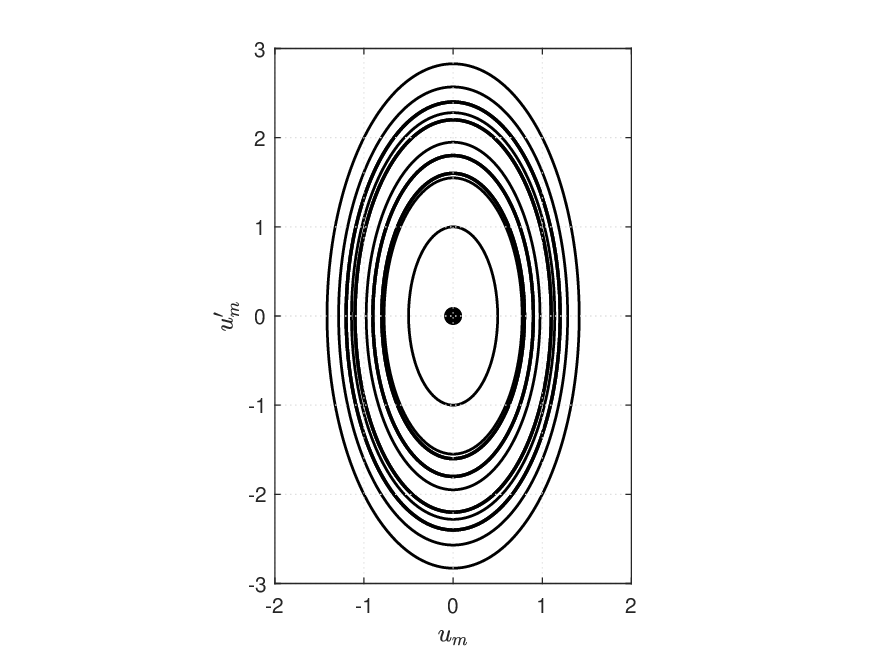}
  	\caption{Phase portrait of the  critical equation ($q_0\leq\lambda_*<q_0+m^2\pi^2$).}
  	\label{fig:img1}
  \end{minipage}
  \hfill
 \begin{minipage}{0.48\textwidth}
 	\includegraphics[width=\textwidth, height=6cm]{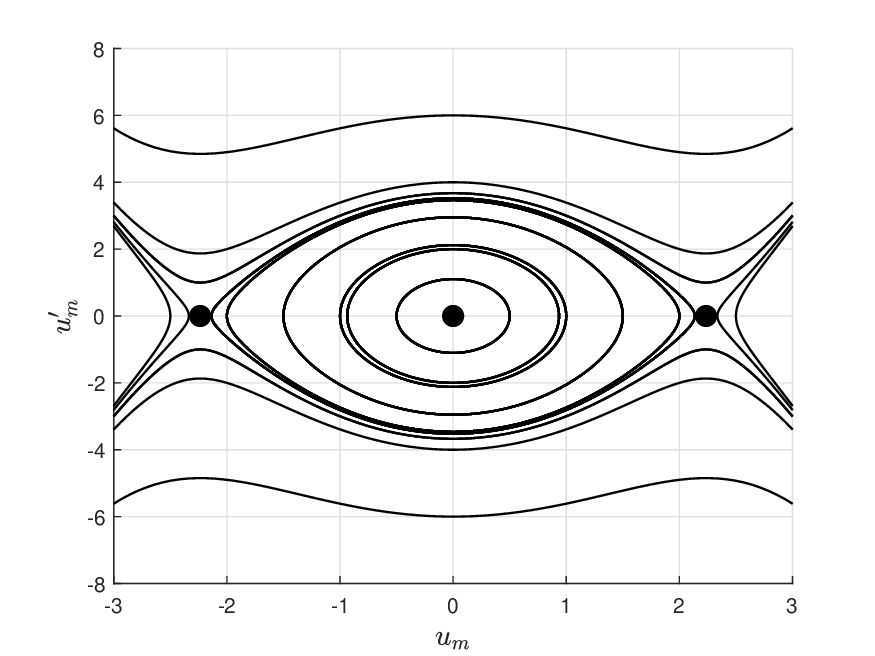}
 	\caption{Phase portrait of the   critical equation ($ q_0+m^2\pi^2<\lambda_*$).}
 	\label{fig:img2}
 \end{minipage}
 \end{figure}
\begin{figure}[ht]
	\centering
	\hspace{1.5cm}\includegraphics[width=8.5cm, height=6cm]{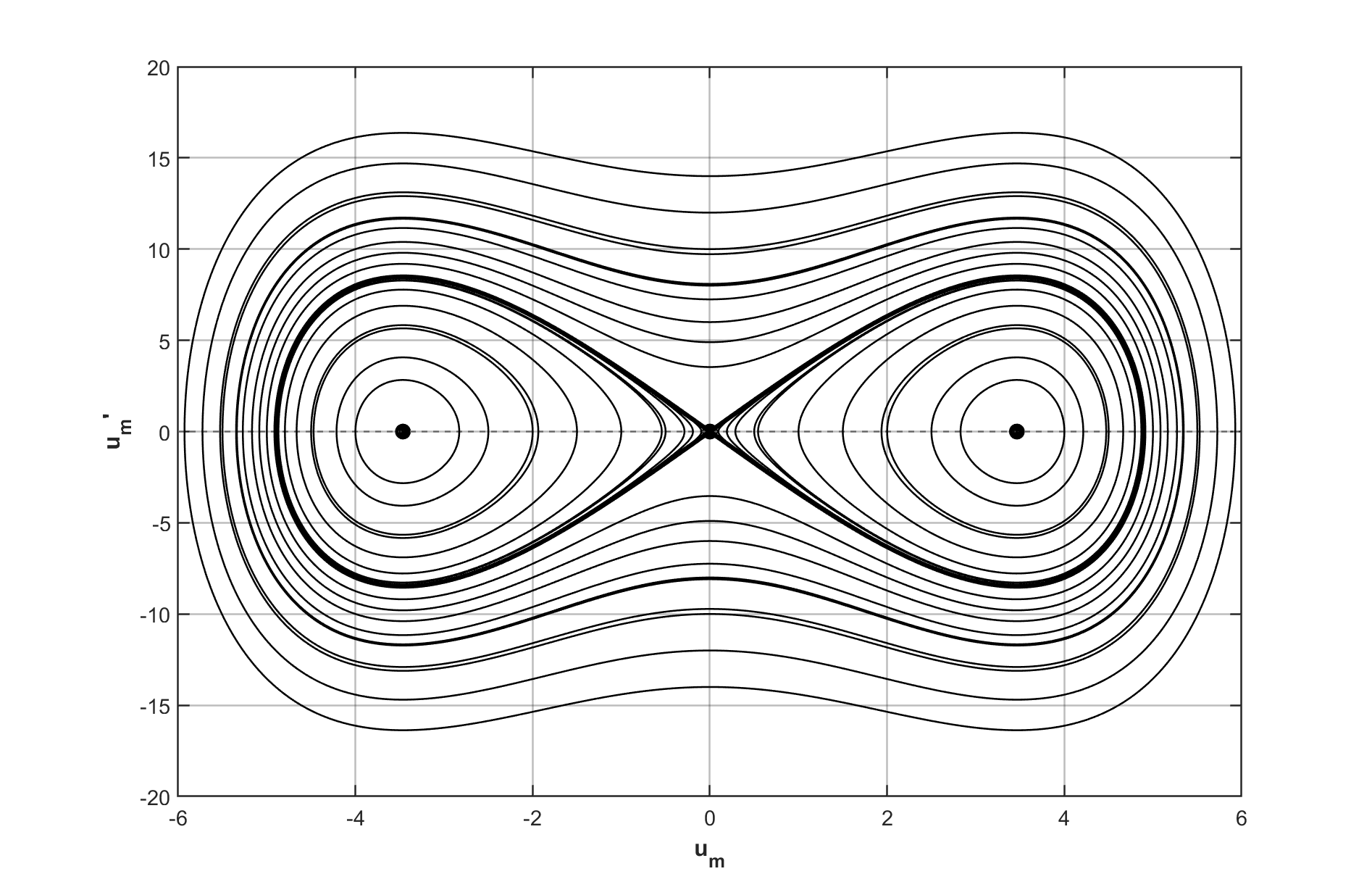}
	\caption{Phase portrait of the  critical equation ($\lambda_*<q_0-m^2\pi^2$).}
	\lb{fig3}
\end{figure}

\begin{theorem}\label{ep_sign}
 {For the case $\lambda_*>q_0$ and $\lambda_*\ne q_0+m^2\pi^2$, the value of $\e$ is given as:} 
 \begin{equation}\label{ep}
  \e=
  \begin{cases}
   +1,&\quad {\rm for}\  q_0+m^2\pi^2<\lambda_*,\\
   -1,& \quad {\rm for}\  q_0<\lambda_*<q_0+m^2\pi^2.
  \end{cases}
 \end{equation} 
\end{theorem}
\begin{proof}
It follows from  $\lambda_*-q_0>0$ that  the solution of \eqref{cri} is a $T$-periodic function.  	
	We introduce the  parameter $a_m$  as
	\[
	a_m\equiv\max_{x\in[0,T]}|u_m(x)|
	\]
and $a_m$ can be solved by 
\begin{equation}\label{3.5}
	-\e\frac{1}{p^*}a_m^{2p^*}+(\lambda_*-q_0)a_m^2=k.         
\end{equation}

For different values of $\epsilon$, the corresponding ranges for the parameter $a_m$ vary accordingly.

If $\epsilon = -1$, it is evident that $a_m \in (0, +\infty)$.\\
In the case of $\e= +1$, it follows from \eqref{3.5}  that 
 \begin{equation}\label{3.2}
 -\frac{1}{p^*}a_m^{2p^*}+(\lambda_*-q_0)a_m^2=k
 \end{equation}
and 
\[
\max_{a_m\in\mathbb R^+}\left((\lambda_*-q_0) a_m^2-\frac{p-1}{p}a_m^{\frac{2p}{p-1}}\right) =\left((\lambda_*-q_0)a_m^2-\frac{p-1}{p}a_m^{\frac{2p}{p-1}}\right)\bigg|_{a_m={(\lambda_*-q_0)^{\frac{p-1}{2}}}}=\frac{(\lambda_*-q_0)^p}{p}.
\]
A necessary and sufficient condition for \eqref{3.2} to be solvable is that $\lambda_*$, $q_0$ and $a_m$ fulfill an additional constraint
\[
0<pk<(\lambda_*-q_0)^p.
\]
We can obtain the range of $a_m$:
\[
a_m\in(0,{(\lambda_*-q_0)}^{\frac{p-1}{2}}).
\]

It follows form \eqref{shouci1} that
\[
u'_m(x)=\sqrt{k+\frac{\e}{p^*}u_m(x)^{2p^*}-(\lambda_*-q_0) u_m^2(x)}, \quad{ \rm for}\quad x\in[0,\frac{T}{4}]. 
\]

Combining \eqref{shouci1} and \eqref{3.2}, we obtain that
\[
\frac{T}{4}=\int_0^{\frac{T}{4}}dx=\int_0^{a_m}\frac{1}{u_m'}du_m.
\]
Set $t=\frac{u_m}{a_m}$, we have
\begin{equation}\label{b}
\frac{1}{2m}=\int_0^1\frac{dt}{\sqrt{\lambda_*-q_0}\sqrt{(1-t^2)-\frac{\e}{p^*(\lambda_*-q_0)}a_m^{2p^*-2}(1-t^{2p^*})}}.
\end{equation}
For $\lambda_*>(m\pi)^2+q_0$, it follows that 
\[
\int_0^1\frac{dt}{\sqrt{(1-t^2)-\e/(p^*(\lambda_*-q_0))a_m^{2p^*-2}(1-t^{2p^*})}}>\frac{\pi}{2},
\]
which implies $\e>0$ and $\e=+1.$ \\
Similarly, we can obtain that  $\e=-1$ for
$0<\la_m^*-q_0<m^2\pi^2$.
\end{proof}
It should be noted that the dynamical behavior of the critical equation is different when $q_0>\lambda_*$.
We present the result for   this case in the following.
\begin{theorem}\label{ep_sign2}
	If $q_0\geq\lambda_*$ in \eqref{cri}, then $\e=-1$. 
\end{theorem}
\begin{proof}
	
	Note that the first integrals of critical equation is \[(u_m'(x))^2-\frac{\e}{p^*}|u_m(x)|^{2p^*}+(\la_m-q_0) u_m^2(x)= k,\quad \text{ for} \ x\in [0,1]\]
At the point where $|u(x)|$ attains its maximum value $a_m$, we have the following equation \[
	-\frac{\e}{p^*}a_m^{2p^*}+(\la_m-q_0) a_m^2= k.
	\]
	From the above two equations, we  further obtain that
	\[(u'_m(x))^2=\frac{\e}{p^*}(u_m(x)^{2p^*}-a_m^{2p^*})+(\lambda_m-q_0)(a_m^2-u_m^2(x))\geq0, \quad \text{for}\ x\in[0,1],
	\]
	which implies that $\e<0$ and $\e=-1$.
\end{proof}

\begin{remark}
In the case where the potential $q$ in \eqref{s-l} reduces to a constant $q_0$, the corresponding eigenvalue is given by $\lambda_m = q_0+m^2\pi^2$, $m=1,2,\cdots$. For the specific case where $m=1$ and $q_0 + m^2\pi^2 < \lambda_*$, the results of Theorem \ref{ep_sign} are consistent with the critical equation established in \cite{V-I}. Conversely, when $q_0 + m^2\pi^2 > \lambda_*$, we find that the value of $\epsilon$ differs, leading to a distinct shift in the dynamical behavior of the critical equation.
\end{remark}
\begin{remark}
	It is noteworthy that the dynamical characteristics of the inverse problem align with the  cases of the direct problem \cite{zhang2}  provided $\lambda_* - q_0 < m^2\pi^2$. However, the condition $\lambda_* - q_0 > m^2\pi^2$ precipitates a previously undocumented phenomenon.
\end{remark}

 According to 
 Theorem \ref{ep_sign} and Theorem \ref{ep_sign2}, the connection between the problem ${ \rm (IOSP)}$ and the nonlinear Schr\"odinger equation \eqref{cri} leads to the expression of $\hat q$.
 \begin{theorem}\label{relation}
 	For   $q_0$, $\lambda_*\in \mathbb R$, the solution $\hat q\in \mathcal{L}^p$ to problem ${ \rm (IOSP)}$ is given by the expression
	\begin{equation}\label{hat q}
	\hat q(x)=\begin{cases}
		q_0+|u_m|^{2p^*-2},\quad {\rm for}\ \lambda_*-q_0>m^2\pi^2,\\
		q_0-|u_m|^{2p^*-2},\quad {\rm for}\ \lambda_*-q_0<m^2\pi^2.
		\end{cases}
	\end{equation}
\end{theorem}
\begin{remark}
Unlike the result $\epsilon = 1$ obtained under the condition $\lambda_* > \lambda_1(q_0)$ in existing literature \cite{V-I},  the value of $\epsilon$ is contingent upon both $q_0$ and $\lambda_*$ in our current setting. 
\end{remark}

\subsection{Establishing exact quantitative relationship between $\lambda_*$, $q_0$ and $\hat q$}
Based on the classification of critical equations established in the previous subsection, the system exhibits three distinct dynamical behaviors depending on the values of $q_0$ and $\lambda_*$, as illustrated in Fig.~\ref{fig:img1}, Fig.~\ref{fig:img2}  and Fig.~\ref{fig3}. We will provide the error $\|\hat q-q_0\|_{\mathcal L^p}$ for the three cases separately.

 For the case where $q_0<\lambda_*<q_0+m^2\pi^2$, the phase portrait of \eqref{cri} is shown in Fig.~\ref{fig:img1}. Let us introduce 
 \begin{equation}\label{V1}
\mathbb V_1(\alpha)= \int_0^1\frac{dt}{\sqrt{(1-t^2)+\frac{1}{p^*(\lambda_*-q_0)}\alpha^{2p^*-2}(1-t^{2p^*})}},\quad \alpha\in(0,+\infty).
 \end{equation}
Since $\mathbb V_1(\alpha)$ is a monotonic function with respect to $\alpha$, we have
\begin{equation}\label{a_m1}
a_m=\mathbb V_1^{-1}\left(\frac{\sqrt{\lambda_*-q_0}}{2m}\right).
\end{equation}
Set \[
\mathbb U_1(\alpha)=\int_0^1\alpha^{2p^*}(\lambda_*-q_0)^{-p}\frac{t^{2p^*}dt}{\sqrt{1-t^2
		+\frac{ \alpha^{2p^*-2}}{p^*(\lambda_*-q_0)}(1-t^{2p^*})}}, \quad \alpha \in(0,+\infty).
\]
Moreover, we obtain
\begin{equation}\label{wucha1}
\begin{split}
\|\hat q-q_0\|_{\mathcal  L^p}=&\left(\int_0^1|\hat q-q_0|^pdx\right)^{\frac{1}{p}}\\
=&\left(2m\int_0^{a_m}\frac{|u_m|^{2p^*}}{u_m'}du_m\right)^{\frac{1}{p}}\\
=&\left(\frac{2ma_m^{2p^*}}{\sqrt{\lambda_*-q_0}}\int_0^1\frac{t^{2p^*}dt}{\sqrt{1-t^2+\frac{ a_m^{2p^*-2}}{p^*(\lambda_*-q_0)}(1-t^{2p^*})}}\right)^{\frac{1}{p}}\\
=&\left(2m(\lambda_*-q_0)^{p-\frac{1}{2}}\mathbb U_1\left(\mathbb V_1^{-1}\left(\frac{\sqrt{\lambda_*-q_0}}{2m}\right)\right)\right)^{\frac{1}{p}}.
\end{split}
\end{equation}
For the case where $q_0+m^2\pi^2<\lambda_*$, the phase portrait of \eqref{cri} is shown in Fig.~\ref{fig:img2}. We make a slight adjustment and set
 \[
\mathbb V_2(\alpha)= \int_0^1\frac{dt}{\sqrt{(1-t^2)-\frac{1}{p^*(\lambda_*-q_0)}\alpha^{2p^*-2}(1-t^{2p^*})}},\quad \alpha\in(0, 1).
\]
Note that $V_2$ is also a monotonic function with respect to $\alpha$, which allows us to conclude that
\begin{equation}\label{a_m2}
a_m=\mathbb V_2^{-1}\left(\frac{\sqrt{\lambda_*-q_0}}{2m}\right)
\end{equation}

Set 
\[
\mathbb U_2(\alpha)=\int_0^1\alpha^{2p^*}(\lambda_*-q_0)^{-p}\frac{t^{2p^*}dt}{\sqrt{1-t^2
		-\frac{ \alpha^{2p^*-2}}{p^*(\la_m-q_0)}(1-t^{2p^*})}}, \quad \alpha \in(0,+\infty).
\]
Moreover, we obtain
\begin{equation}\label{wucha2}
\begin{split}
	\|\hat q-q_0\|_{\mathcal  L^p}=&\left(\int_0^1|\hat q-q_0|^pdx\right)^{\frac{1}{p}}\\
	=&\left(2m\int_0^{a_m}\frac{|u_m|^{2p^*}}{u_m'}du_m\right)^{\frac{1}{p}}\\
	=&\left(\frac{2ma_m^{2p^*}}{\sqrt{\lambda_*-q_0}}\int_0^1\frac{t^{2p^*}dt}{\sqrt{1-t^2-\frac{ a_m^{2p^*-2}}{p^*(\lambda_*-q_0)}(1-t^{2p^*})}}\right)^{\frac{1}{p}}\\
	=&\left(2m(\lambda_*-q_0)^{p-\frac{1}{2}}\mathbb U_2\left(\mathbb V_2^{-1}\left(\frac{\sqrt{\lambda_*-q_0}}{2m}\right)\right)\right)^{\frac{1}{p}}.
\end{split}  
\end{equation}
For the case where $\lambda_*<q_0$, the phase portrait of \eqref{cri} is shown in Fig.~\ref{fig3}. Noting that the solution is $T$-periodic,  we  consider the  parameter
\[
a_m\equiv\max_{x\in[0,T]}|u_m(x)|\in[(p^*(q_0-\lambda_*))^{\frac{p-1}{2}},+\infty).
\]
For $x\in[0,\frac{T}{4}]$, it follows form \eqref{shouci1} that
\[\begin{split}
u'_m(x)&=\sqrt{k-\frac{1}{p^*}u_m(x)^{2p^*}-(\lambda_*-q_0) u_m^2(x)}\\
&=\sqrt{\frac{1}{p^*}(a_m^{2p^*}-u_m(x)^{2p^*})-(q_0-\lambda_m)(a_m^2-u_m^2(x))}. 
\end{split}
\]
Note that  
\[
\begin{split}
\frac{T}{4}=\int_0^{\frac{T}{4}}dx=&\int_0^{a_m}\frac{1}{u_m'(x)}du_m(x)\\
=&\int_0^{a_m}\frac{du_m}{\sqrt{\frac{1}{p^*}(a_m^{2p^*}-u_m^{2p^*})-(q_0-\lambda_*)(a_m^2-u_m^2)}}.
\end{split}
\]
Moreover, by setting $t=\frac{u_m}{a_m}$, we obtain that
\[
\frac{1}{2m}=\int_0^1\frac{dt}{\sqrt{q_0-\lambda_*}\sqrt{\frac{a_m^{2p^*-2}}{p^*(q_0-\lambda_*)}(1-t^{2p^*})-(1-t^2)}}.
\]
Now we introduce 
\[
\mathbb V_3(\alpha)=\int_0^1\frac{dt}{\sqrt{\frac{\alpha^{2p^*-2}}{p^*(q_0-\lambda_*)}(1-t^{2p^*})-(1-t^2)}},\quad {\alpha\in[(p^*(q_0-\lambda_*))^{\frac{p-1}{2}},+\infty)}
\]
and
\[
\mathbb U_3(\alpha)=\alpha^{2p^*}(q_0-\lambda_*)^{-p}\int_0^1\frac{t^{2p^*}dt}{\sqrt{\frac{\alpha^{2p^*-2}}{p^*(q_0-\lambda_*)}(1-t^{2p^*})-(1-t^2)}}, \alpha\in(0,+\infty).
\]
Since $\mathbb V_3(\alpha)$ is also a monotonic function with respect to $\alpha$, we have \begin{equation}\label{a_m}
a_m=\mathbb V_3^{-1}\left( \frac{\sqrt{q_0-\lambda_*}}{2m}\right).
\end{equation}
	Note that $V_i(\alpha)$ ($i=1,2,3$) are monotonic functions of $\alpha$, and hence invertible.
Moreover,
\begin{equation}\label{wucha3}
\begin{split}
	\|\hat q-q_0\|_{\mathcal{L}^p}&=\left(\int_0^1|q-q_0|dx\right)^{\frac{1}{p}}\\
	&=\left(2m\int_0^{a_m}\frac{|u_m|^{2p^*}}{u_m'}du_m\right)^{\frac{1}{p}}\\
	&=\left(2m(q_0-\lambda_*)^{p-\frac{1}{2}}\mathbb  U_3\left(\mathbb V_3^{-1}\left(\frac{\sqrt{q_0-\lambda_*}}{2m}\right)\right)\right)^{\frac{1}{p}}.
\end{split}
\end{equation}

\begin{remark}
	To systematically provide quantitative results for the error in problem { \rm (IOSP)}, we further elaborate on the cases where $\lambda_* = q_0$ and $\lambda_* = q_0 + m^2 \pi^2$.	
	For $\lambda_* - q_0 = m^2 \pi^2$, it is straightforward to show that the solution to problem { \rm (IOSP)} is $\hat q = q_0$ and the error $\|\hat q-q_0\|_{\mathcal L^p}=0$.		
	For the case $\lambda_*=q_0$, we obtain $\e=-1$.   Repeat the steps from \eqref{V1} to \eqref{wucha1}, we obtain \[
	\|\hat q-q_0\|_{\mathcal L^p}=\left(\frac{4m^2}{{p^*}}\right)^p\left(\int_0^1\frac{1}{\sqrt{1-t^{2p^*}}}\right)^{2p-1}\int_0^1\frac{t^{2p^*}}{\sqrt{1-t^{2p^*}}}dt.
	\] 
	The error  is continuous with respect to $\lambda_*$, as 
	\[
\lim_{\lambda_*\to q_0+{m^2\pi^2}^+}	\|\hat q-q_0\|_{\mathcal L^p}=\lim_{\lambda_*\to q_0+{m^2\pi^2}^-}	\|\hat q-q_0\|_{\mathcal L^p}=\|\hat q-q_0\|_{\mathcal L^p}\bigg|_{\lambda_*=m^2\pi^2}.
\] 
	and
	\[
	\lim_{\lambda_*\to q_0^+}	\|\hat q-q_0\|_{\mathcal L^p}=\lim_{\lambda_*\to q_0^-}	\|\hat q-q_0\|_{\mathcal L^p}=\|\hat q-q_0\|_{\mathcal L^p}\bigg|_{\lambda_*=q_0}.
	\] 
\end{remark}

 {	
The functions $\mathbb V_1, \mathbb V_2$ and $\mathbb V_3$ serve as a bridge to derive the relationship between the parameter $\alpha_m$  with $q_0$ and $\lambda_*$. 
	Thus, the infinite-dimensional variational problem investigated in this paper is converted to a one-dimensional nonlinear equation regarding the characteristic parameter $\alpha_m$. Once the value of $\alpha_m$ is determined, the  relationships among $\hat{q}$, $q_0$, and $\lambda_*$ can be explicitly established via  \eqref{wucha1}, \eqref{wucha2} and \eqref{wucha3}, as detailed in the following theorem.}

For Problem ${ \rm (IOSP)}$, we provide the following quantitative error $\|\hat q-q_0\|_{\mathcal{L}^p}$ for all observed values $q_0$, $\lambda_*$.
\begin{theorem}
	\label{Th6}
Given the constants $q_0$ and $\lambda_*\in \mathbb R$, the quantitative relationship between the solution $\hat{q}$,  $q_0$ and  $\lambda_*$ is established through the error $\|\hat q-q_0\|_{\mathcal{L}^p}$  as 
\begin{equation}\label{result}
	\|\hat q-q_0\|_{\mathcal{L}^p}=\begin{cases}
		\left(2m(\lambda_*-q_0)^{p-\frac{1}{2}}\mathbb U_2\left(\mathbb V_2^{-1}\left(\frac{\sqrt{\lambda_*-q_0}}{2m}\right)\right)\right)^{\frac{1}{p}},\quad &{ \rm for} \ \lambda_*>q_0+m^2\pi^2,\\
		0,\quad &{ \rm for } \ \lambda_*=q_0+m^2\pi^2,\\
		\left(2m(\lambda_*-q_0)^{p-\frac{1}{2}}\mathbb U_1\left(\mathbb V_1^{-1}\left(\frac{\sqrt{\lambda_*-q_0}}{2m}\right)\right)\right)^{\frac{1}{p}},\quad &{ \rm for} \ q_0<\lambda_*<q_0+m^2\pi^2,\\
		\left(\frac{4m^2}{{p^*}}\right)^p\left(\int_0^1\frac{1}{\sqrt{1-t^{2p^*}}}\right)^{2p-1}\int_0^1\frac{t^{2p^*}}{\sqrt{1-t^{2p^*}}}dt,  \quad &{ \rm for} \ \lambda_*= q_0,\\
		\left(2m(q_0-\lambda_*)^{p-\frac{1}{2}}\mathbb U_3\left(\mathbb V_3^{-1}\left(\frac{\sqrt{q_0-\lambda_*}}{2m}\right)\right)\right)^{\frac{1}{p}}, \quad &{ \rm for} \ \lambda_*< q_0.
	\end{cases}
\end{equation}
\end{theorem}

\subsection{Dynamical system framework yields a sharper result for $ {q} \in \mathcal{L}^2$}
Recall that in Section 3, we established the relationship between $\hat{q}$, $q_0$, and $\lambda_*$ via the error $\|\hat q-q_0\|_{\mathcal L^p}$. In this section, we revisit this relationship for the inverse spectral problem of the principal eigenvalue, specifically under the condition $q \in \mathcal{L}^2$. By employing elliptic functions, the dependence of the error $\|\hat q-q_0\|_{\mathcal L^p}$ on $\lambda_*$ and $q_0$ becomes more transparent, which further supports the numerical results presented in Fig.~\ref{fig5}.

Set $A = \frac{\alpha^2}{2(\lambda_*-q_0)}$, $B = \frac{\alpha^2}{2(q_0-\lambda_*)}$.
According to the previous section, the following functional expressions can be derived in terms of elliptic functions.

\[\mathbb V_1(\alpha) = \frac{1}{\sqrt{1+A}} \mathbb{E}_1\left( \frac{-A}{1+A}\right),\]

\[\mathbb U_1(\alpha) = \frac{4A}{\sqrt{1+A}} \left[ \mathbb{E}_2\left( \frac{-A}{1+A} \right) - \mathbb{E}_1\left( \frac{-A}{1+A} \right) \right],\]
\[\mathbb V_2(\alpha) = \frac{1}{\sqrt{1-A}} \mathbb{E}_1\left( \frac{A}{1-A}\right),\]

\[\mathbb U_2(\alpha) = \frac{4A}{\sqrt{1-A}} \left[ \mathbb{E}_1\left( \frac{A}{1-A} \right) - \mathbb{E}_2\left( \frac{A}{1-A} \right) \right],\]

\[\mathbb V_3(\alpha) = \frac{1}{\sqrt{B-1}} \mathbb{E}_1\left( \frac{-B}{B-1} \right),\]

\[\mathbb U_3(\alpha) = \frac{4B}{\sqrt{B-1}} \left[ \mathbb{E}_2\left( \frac{-B}{B-1} \right) - \mathbb{E}_1\left( \frac{-B}{B-1} \right) \right],\]

	where 
	\[
	\begin{split}
		\mathbb E_1(s)&=\int_0^1\frac{1}{\sqrt{(1-t^2)(1-st^2)}}dt,
		\\
		\mathbb E_2(s)&=\int_0^1\frac{\sqrt{1-st^2}}{\sqrt{1-t^2}}dt.
	\end{split}
	\]
	
	Moreover,  we can numerically investigate the relationship between $ \lambda_*$, $q_0$ and $\|\hat q-q_0\|_{\Lp}$ (see Fig. \ref{fig5} for the case $m=1$, $p=2$).
	\begin{figure}[ht]
		\centering
			\hspace{1.5cm}\includegraphics[width=8.5cm, height=6cm]{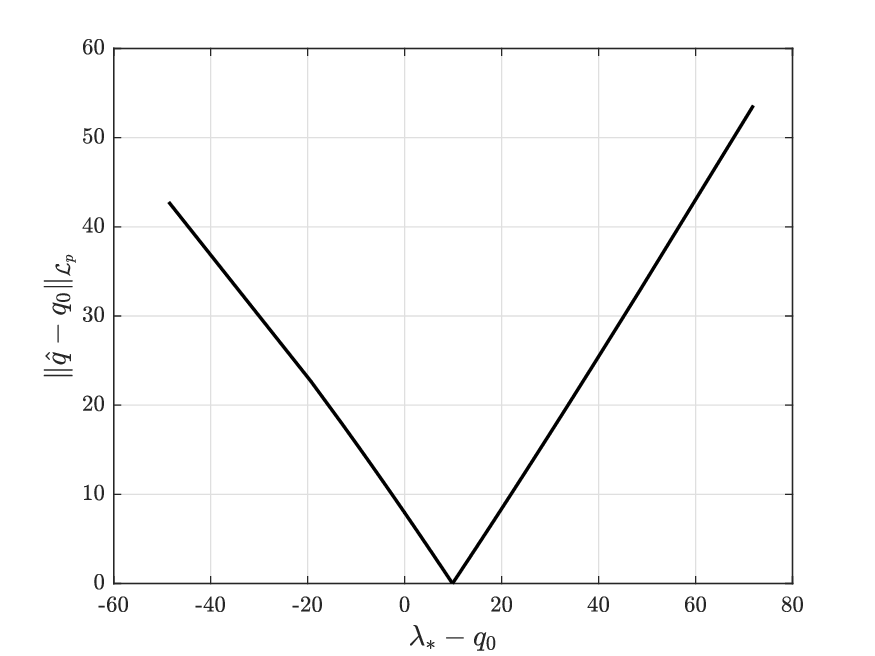}
		\caption{The relationship between  $\lambda_*-q_0$ and $\|\hat q-q_0\|_{\mathcal{L}^p}$ $(p=2)$.}
		\lb{fig5}
	\end{figure}
	
	Furthermore, we reconstruct the explicit analytical solution $u_1$ and the potential $\hat q$.
	
	For $\lambda_*-q_0>\pi^2$, by integrating the Equation \eqref{cri}, we obtain 
	\begin{equation}\label{xishu1}
		\begin{split}
			(u_1')^2&=\frac{1}{2}u_1^4-(\lambda_*-q_0)u_1^2+(\lambda_*-q_0)a_1^2-\frac{1}{2}a_1^4\\
			&=\frac{1}{2}(a_1^2-u^2)[2(\lambda_*-q_0)-a_1^2-u_1^2]>0
		\end{split}
	\end{equation}
	and 
	\[
	\lambda_*-q_0>\frac{a_1^2+u_1^2}{2}\geq\frac{a_1^2}{2},
	\]
	where 
	\[
	a_1=\max_{x\in[0,1]}|u_1(x)|.
	\]
	Let us introduce the Jacobi elliptic function to express the solution of the equation $u_1(x)=a_1 sn(\omega_1 x;k_1)$, where ${ sn}:(\mathbb R; [0,1])\to [0,1]$,  $sn(z_1;z_2)=\sin\psi$ and   \[
	z_1=\int_0^{\psi}\frac{d\theta}{\sqrt{1-z_2^2\sin^2\theta}}.
	\]
	Now, we specify the values for the parameters $\omega_1$ and $k_1$.
	Note that \[
	u_1'=a_1\omega_1\cdot \text{cn}(\omega_1 x;k_1)\cdot \text{dn}(\omega_1 x;k_1),
	\]
	where 
	\[
	\begin{split}
		\text{cn}^2(\omega_1 x;k_1)+\text{sn}^2(\omega_1 x; k_1)&=1,\\
		k_1^2\text{sn}^2(\omega_1 x;k_1)+\text{dn}^2(\omega_1 x;x)&=1.
	\end{split}
	\]
	It follows that  
	\begin{equation}\label{xishu2}
		\begin{split}
			(u_1')^2=&a_1^2\omega_1^2(1-{\rm sn}^2(\omega_1 x;k_1))\cdot(1-k_1^2{\rm sn}^2(wx;k_1))\\
			=&\frac{\omega_1^2k_1^2}{a_1^2}(a_1^2-u_1^2)\left(\frac{a_1^2}{k_1^2}-u_1^2\right).
		\end{split}
	\end{equation}
	Combining \eqref{xishu1} with \eqref{xishu2}, we obtain that 
	\begin{equation}
		\begin{split}
			k_1 &= \sqrt{\frac{a_1^2}{2(\lambda_* - q_0) - a_1^2}} = \sqrt{\frac{\left[ \mathbb{V}_2^{-1} \left( \frac{\sqrt{\lambda_* - q_0}}{2} \right) \right]^2}{2(\lambda_* - q_0) - \left[ \mathbb{V}_2^{-1} \left( \frac{\sqrt{\lambda_* - q_0}}{2} \right) \right]^2}},\\
			\omega_1&=\sqrt{\lambda_*-q_0-\frac{a_1^2}{2}}= \sqrt{\frac{2(\lambda_* - q_0) - \left[ \mathbb{V}_2^{-1} \left( \frac{\sqrt{\lambda_* - q_0}}{2} \right) \right]^2}{2}}.
		\end{split}
	\end{equation}
	It follows that \[u(x)=
	\pm \mathbb{V}_2^{-1} \left(\frac{\sqrt{\lambda_*-q_0}}{2}\right)\text{sn}(\omega_1x,k_1), \quad { \rm for} \quad \lambda_* > q_0 + \pi^2.
	\]
	
	For the case $\lambda_*-q_0<\pi^2$, similarly it follows from \eqref{cri} that 
	\[
	(u_1')^2=\frac{1}{2}(a_1^2-u_1^2)\left(\frac{a_1^2+u_1^2}{2}+\lambda_*-q_0\right).
	\]
	Set $u_1(x)=a_1\text{cn}\left( 2\mathbb E_1(k)x - \mathbb E_1(k); k \right)$, where ${ \text{cn}}:(\mathbb R; [0,1])\to [0,1]$,  $\text{cn}(z_1; z_2)=\cos\psi$ and  \[
	z_1=\int_0^{\psi}\frac{d\theta}{\sqrt{1-z_2^2\sin^2\theta}}.
	\]
	It follows that 
	\[
	u_1''(x) = -4\mathbb E_1(k) ^2(1 - 2k^2) u_1 - \frac{8\mathbb E_1(k)^2 k^2}{a_1^2} u_1^3.
	\]
	Moreover, 
	\[
	k=\sqrt{\frac{a_1^2}{2(\lambda_*-q_0)+2a_1^2}}=\begin{cases}
		k_2=\sqrt{\frac{ \mathbb V_1^{-1}\left(\frac{\sqrt{\lambda_*-q_0}}{2}\right)}{2(\lambda_*-q_0)+2 \mathbb V_1^{-1}\left(\frac{\sqrt{\lambda_*-q_0}}{2}\right)^2}}, \quad \text{for}\  q_0<\lambda_*<q_0+\pi^2,\\
		k_3=	\sqrt{\frac{ \mathbb V_3^{-1}\left(\frac{\sqrt{q_0-\lambda_*}}{2}\right)}{2(\lambda_*-q_0)+2 \mathbb V_3^{-1}\left(\frac{\sqrt{q_0-\lambda_*}}{2}\right)^2}},\quad\text{for}\  \lambda_*<q_0.
	\end{cases}.
	\]
	Now we give the expression for $u_1$:
	\begin{equation}
		u_1=\begin{cases}
			\pm \mathbb V_1^{-1}\left(\frac{\sqrt{\lambda_*-q_0}}{2}\right)\text{cn}\left(2\mathbb E_1(k_2)x - \mathbb E_1(k_2); k_2 \right),\quad\text{for}\   q_0<\lambda_*<q_0+\pi^2,\\
			\pm \mathbb V_3^{-1}\left(\frac{\sqrt{\lambda_*-q_0}}{2}\right)\text{cn}\left(2\mathbb E_1(k_3)x - \mathbb E_1(k_3); k_3 \right),\quad\text{for}\   \lambda_*<q_0.
		\end{cases}
	\end{equation}
	For the case where $\lambda_*=q_0$, we have 
	\[
	u_1(x) = k_4 \cdot \frac{1}{\sqrt{2}} \cdot \frac{\text{sn}(\sqrt{2}k_4x; \frac{1}{\sqrt{2}})}{\text{dn}(\sqrt{2}k_4x; \frac{1}{\sqrt{2}})},
	\]
	where
	\[
	k_4=2\int_0^1\frac{1}{\sqrt{\frac{1}{2}(1-t^{4})}}dt.
	\]
	Furthermore, based on \eqref{hat q}, we obtain the following theorem.
	\begin{theorem}\label{3.10}
		\label{Th6.1}
		For $q_0, \lambda_* \in \mathbb{R}$, $m=1$ and $p=2$, the solution to Problem (IOSP) can be explicitly expressed as follows
		\[
		\hat q=\begin{cases}
			q_0+\left(\mathbb{V}_2^{-1} \left(\frac{\sqrt{\lambda_*-q_0}}{2}\right)\rm{sn}(\omega_4x;k_1)\right)^2,  \quad&\rm{for}\   \lambda_* > q_0 + \pi^2,\\
			q_0,\quad &\rm{for}\ \lambda_*=q_0+\pi^2,\\
			q_0-\left( \mathbb V_1^{-1}\left(\frac{\sqrt{\lambda_*-q_0}}{2}\right){\rm cn}\left(2\mathbb E_1(k_2)x - \mathbb E_1(k_2); k_2 \right)\right)^2,\quad &{\rm for}\  q_0<\lambda_*<q_0+\pi^2,\\
			q_0-\left( \mathbb V_3^{-1}\left(\frac{\sqrt{\lambda_*-q_0}}{2}\right){\rm cn}\left(2\mathbb E_1(k_3)x - \mathbb E_1(k_3); k_3 \right)\right)^2,\quad &{\rm for}\  \lambda_*<q_0,\\
			q_0- \left(\frac{k_4}{\sqrt{2}} \frac{{\rm sn}(\sqrt{2}k_4x; \frac{1}{\sqrt{2}})}{{\rm  dn}(\sqrt{2}k_4x; \frac{1}{\sqrt{2}})}\right)^2,\quad &{\rm for}\ \lambda_*=q_0.
		\end{cases}
		\]
	\end{theorem}
	\begin{remark}
		Recall that the expression for $\|\hat q - q_0\|_{\mathcal L^p}$ derived in Section 3 explicitly characterizes the relationship between $\hat q$, $q_0$, and $\lambda_*$. Geometrically, this implies that $\hat q$ is confined to a fixed sphere within the $\mathcal L^p$ space.
		However, the special case where ${q}\in\mathcal{L}^2$ yields a significantly sharper result: an exact analytical expression for $\hat{q}$. This is to say, Problem (IOSP) is completely resolved in the ${q}\in\mathcal{L}^2$ setting.  Such a result serves as a foundational tool with substantial implications for various problems in mathematical physics and engineering.
		
	\end{remark}

\subsection{Uniqueness of the optimal potential $\hat{q}$}
Consistent with the previous planar dynamical systems method, the relationship among $\hat q$, $\lambda_*$, and $q_0$ is constructed via the error $\|\hat q-q_0\|_{\mathcal L^p}$. Furthermore, $\hat q$ is defined as a unique function of $\lambda_*$ and $q_0$.
\begin{theorem}\label{3.12}
	Given any $\lambda_*$,  $q_0\in \mathbb R$, problem { \rm (IOSP)} admits a unique solution $\hat q$.
\end{theorem}
\begin{proof}
	For the case $\lambda_*-q_0= m^2\pi^2$, the uniqueness of $\hat q$ is evident. It follows that $\|\hat q-q_0\|_{\mathcal L^p}=0$ and $\hat q = q_0$.
	
	For the case 	$\lambda_*-q_0\ne m^2\pi^2$, the solution to \eqref{cri} exhibits periodicity. The maximum value of the $u_m(x)$ over a period is uniquely determined by \eqref{a_m1}, \eqref{a_m2} and \eqref{a_m}. 
	\begin{equation}\label{zuidazhi}
		u_m\left(\frac{T}{4}\right)=	u_m\left(\frac{m}{2}\right)=\pm\max_{x\in[0,T]}|u_m(x)|=\begin{cases}
			\pm \mathbb  V_1^{-1}\left(\frac{\sqrt{\lambda_*-q_0}}{2m}\right),\quad\text{for}\  q_0<\lambda_*<q_0+m^2\pi^2,\\
			\pm \mathbb V_2^{-1}\left(\frac{\sqrt{\lambda_*-q_0}}{2m}\right),\quad\text{for}\  q_0+m^2\pi^2<\lambda_*,\\
			\pm\left(2m\int_0^1\frac{1}{\sqrt{\frac{1}{p^*}(1-t^{2p^*})}}dt\right)^{p-1},\quad \text{for}\ \lambda_*=q_0,\\
			\pm \mathbb  V_3^{-1}\left( \frac{\sqrt{q_0-\lambda_*}}{2m}\right),\quad \text{for}\ \lambda_*<q_0.
		\end{cases}
	\end{equation} 
	Based on \eqref{zuidazhi},  $u_m'\left(\frac{T}{4}\right)=0$ and the existence and uniqueness theorem, we know that \eqref{cri} admits a unique pair of non-trivial solutions, $u_m$ and $-u_m$ (i.e., the solution is unique up to a sign change).\\
	Moreover, according to \eqref{hat q}, $\hat q$ is uniquely defined by $\lambda_*$ and $q_0$.
\end{proof}

\begin{remark}
	While the uniqueness of the inverse spectral problem  was studied in \cite{V-I} subject to $\lambda_* > \lambda_1(q_0)$, our work dispenses with this requirement. By utilizing a planar dynamical systems method, we prove the uniqueness for arbitrary parameters $(\lambda_*, q_0)$. In contrast, previous research on uniqueness has been heavily dependent on variational methods and convex analysis. To ensure this convexity, the restriction $\lambda_* > \lambda_1(q_0)$ becomes indispensable.
\end{remark}

\section{Dilation relation}

Note that the traditional inverse optimization spectral problem was described 
as follows: 
	\begin{equation}\label{old}
		\|\hat q-q_0\|_{\mathcal L^p}=\min\{\|q-q_0\|_{\mathcal L^p}:\lambda_*={\lambda_{\bf 1}(q)},\ q\in \mathcal L^p\}.
	\end{equation}
	Notably, \eqref{old} lacks flexibility for practical applications, as it rigidly restricts the observed eigenvalue to $\lambda_1(q)$. Our formulation \eqref{youhua}, however, imposes no such restriction, allowing the observed eigenvalue to be $\lambda_m(q)$ for any arbitrary $m$.
	$\la_1(q) = \lambda_*$ is in fact a specific case of $\la_m(q)=\lambda_*$ for any arbitrary $m$.
Interestingly, the method we employ in this work to determine the sign of the higher-order term in the critical equation departs from that of \cite{V-I}. Our alternative approach is not only applicable to the specific case of $\la_1(q) = \lambda_*$ treated in \cite{V-I} but also remains valid for the general case where $\la_m(q)=\lambda_*$ for any arbitrary $m$. In fact, whether the observed eigenvalue $\lambda_*$ is the principal eigenvalue or the $m$-th eigenvalue, the potential  can be reconstructed, and the error can be directly established via a homeomorphic mapping.  Let us define  the error $\mathcal Z_m$:
\begin{equation}
	\mathcal Z_m(x)=\begin{cases}
		\left(2mx^{p-\frac{1}{2}}\mathbb U_2\left(\mathbb V_2^{-1}\left(\frac{\sqrt{x}}{2m}\right)\right)\right)^{\frac{1}{p}},\quad &{ \rm for} \ x>m^2\pi^2,\\
		0,\quad &{ \rm for}\ x=m^2\pi^2,\\
		\left(2mx^{p-\frac{1}{2}}\mathbb U_1\left(\mathbb V_1^{-1}\left(\frac{\sqrt{x}}{2m}\right)\right)\right)^{\frac{1}{p}},\quad &{ \rm for}\ 0<x<m^2\pi^2,\\
			\left(\frac{4m^2}{{p^*}}\right)^p\left(\int_0^1\frac{1}{\sqrt{1-t^{2p^*}}}\right)^{2p-1}\int_0^1\frac{t^{2p^*}}{\sqrt{1-t^{2p^*}}}dt,  \quad &{ \rm for} \ x= 0,\\
		\left(2m(-x)^{p-\frac{1}{2}}\mathbb U_3\left(\mathbb V_3^{-1}\left(\frac{\sqrt{-x}}{2m}\right)\right)\right)^{\frac{1}{p}}, \quad &{ \rm for}\ x< 0,
	\end{cases}
\end{equation}
and 
\[
\mathcal R_m(x)=\frac{x}{m^2},\quad { \rm for}\ x\in\mathbb R.
\]
Consequently, we establish the following dilation relation theorem.
\begin{theorem}
{\bf(Dilation relation)}	For $m\in \mathbb N$,  $p\in(1,+\infty]$, there holds the following dilation relation between error  of higher order and the principal eigenvalues.
	\[
	\mathcal Z_1\circ \mathcal R_m(x)=\mathcal R_m\circ\mathcal Z_m(x),\quad { \rm for}\ x\in \mathbb R.
	\]
Moreover, for $\lambda_1=\lambda_*$, $\lambda_m=\lambda_*$, $\lambda_*\in \mathbb R$, the error corresponding to the inverse spectral problem for the $m$-th eigenvalue is expressed as follows
\[
\|\hat q-q_0\|_{\mathcal L^p}=\mathcal Z_m(\lambda_m-q_0)=\mathcal R_m^{-1}\circ\mathcal Z_1\circ \mathcal R_m(\lambda_1-q_0).
\]
\end{theorem}

{\begin{remark}
	
This homeomorphism demonstrates that our dynamical-system framework applies to Problem (ISOP) for any \(\lambda_m(q)\) as the observed eigenvalue \(\lambda_*\), without being restricted to \(\lambda_1(q)\). Conventional inverse spectral work, however, is only valid under the constraint \(\lambda_*=\lambda_1(q)\).
\end{remark}
}


\section{Conclusions and practical applications}


 \subsection{Theoretical conclusions }
 
 This paper establishes a dynamical-system framework that yields quantitative results for the inverse optimal spectral problem. The primary contributions and innovations of this study are summarized as follows:
 
 {
 (i) {\em Quantitative results:} By introducing  planar dynamical system approach, we initially establish a quantitative results for the inverse spectral problem. We convert an infinite-dimensional optimization problem into a one-dimensional, exactly quantified problem, which quantitatively characterizes the relationship between the reconstructed potential $\hat{q}$, its target potential  $q_0$, and the observed eigenvalue $\lambda_*$. 
Remarkably, when $q \in \mathcal{L}^2$, our approach delivers an exact analytical expression for $\hat{q}$ (see Theorem \ref{Th6.1}), thereby completely resolving the optimization inverse spectral problem. Note that previous literature is confined to the qualitative analysis (existence and uniqueness), which is not applicable to the practical applications in structural health monitoring and damage detection, material design, seismic wave analysis, sonar detection, and related fields.

(ii) {\em Eliminating a long-standing constraint:}  { This work eliminates a long-standing constraint in the inverse optimization spectral problems: the requirement $\lambda_* > \lambda_1(q_0)$ found in previous studies.
	This constraint arises because previous results are heavily reliant on convex analysis theory. To ensure this convexity, the restriction $\lambda_* > \lambda_1(q_0)$ becomes indispensable.
	However, this imposed condition is often impractical in real-world scenarios. In this paper, we establish a comprehensive  framework valid for all $\lambda_*, q_0 \in \mathbb{R}$, extending the theory to a global parameter space. 	
	}

(iii) {\em The superiority of the method:}  In a decisive step beyond classical techniques, we introduce a dynamical system approach that fundamentally transforms the problem. Previous research on this problem has primarily relied on variational methods and convex analysis, yielding only qualitative results such as existence and local uniqueness. By contrast, our dynamical system approach not only provides an explicit quantitative interplay between the reconstructed potential $\hat{q}$, the baseline $q_0$, and the eigenvalue $\lambda_*$, but also eliminates the unnecessary constraint $\lambda_* > \lambda_1(q_0)$.

(iv) {\em Homeomorphic mapping:}  An extra contribution of this work is the construction of a homeomorphic mapping, which enables us to determine the dilation relation linking the errors $\|\hat{q} - q_0\|_{\mathcal{L}^p}$ associated with $\lambda_1(q)$ and $\lambda_m(q)$.
This pivotal result enables the extension of error findings from the $m$-th eigenvalue to the principal eigenvalue problem 
thereby establishing a unified quantitative approach to solving different classes of inverse spectral problems. {This homeomorphism shows our dynamical-system framework accommodates any \(\lambda_m(q)\) as \(\lambda_*\) in Problem (ISOP), rather than only \(\lambda_1(q)\). By contrast, classical inverse spectral analyses are confined to the case \(\lambda_*=\lambda_1(q)\).
}

}

\subsection{Practical applications}

The investigation of the inverse spectral problem (IOSP) presented in this paper is of considerable practical significance. With the theoretical framework established, we now proceed to address the four questions raised in the introduction:

1. In the context of structural health monitoring, the stiffness distribution of critical components such as bridge stay cables or architectural beams is modeled as $q_0$. When structural integrity is compromised by cracks or corrosion, the resulting shifts in the measured resonant frequency $\lambda_*$ serve as a diagnostic indicator. By utilizing the observed frequency $\lambda_*$, Our framework reconstructs the current physical state $\hat{q}$, enabling precise localization of the most probable sites of structural damage by computing the error $\|\hat{q} - q_0\|_{\mathcal{L}^p}$ as described in Theorem \ref{Th6} or by using the explicit expression of $\hat{q}$ provided in Theorem \ref{Th6.1}. If $\|\hat{q} - q_0\|_{\mathcal{L}^p}$ tends to $0$, the structure is in good condition. In particular, if $\hat{q} = q_0$, the structure is perfectly sound. {Convex analysis approaches impose a constraint that the observed frequency must exceed the fundamental frequency of the structure, namely \(\lambda_*>\lambda_1(q_0)\), which is difficult to satisfy in practical engineering applications (consequently, not reasonable).  By contrast, the results derived in this paper are valid for arbitrary \(\lambda_*\) and \(q_0\). Furthermore, traditional inverse spectral analysis is only applicable to the fundamental frequency \(\lambda_1\); a unified theoretical framework for damage detection using higher-order modes is absent. Multiple orders of vibration frequencies are generally obtainable from field measurements of bridges, yet such multi-frequency data cannot be fully exploited by existing approaches. This paper constructs a homeomorphic mapping to establish a scaling relationship between the $m$-th eigenvalue and the principal eigenvalue. It makes full use of multi-order frequency data collected from bridge vibration monitoring, laying a solid theoretical foundation for bridge structural health monitoring based on various vibration frequencies.}

2. {When dealing with a rare string that has suffered centuries of physical wear, resulting in a drifted timbre, how can a restorer apply minimal perturbations to the initial physical state $q_0$ to return the string to its original, ideal resonance $\lambda_*$?
Based on the results of this paper, if the material of the string is approximately homogeneous (i.e. $q_0$ is a constant), the restoration solution $\hat{q}$ is unique, and the structural change $\|\hat{q} - q_0\|_{\mathcal{L}^p}$ resulting from the restoration can be determined. Moreover, for $q \in \mathcal{L}^2$, we can derive the exact expression for the restoration solution $\hat{q}$.}

3. In seismic wave studies, $\lambda^*$ represents seismic eigenfrequencies, while $q_0$ denotes an initial estimate of the subsurface structure $q(x)$ derived from prior knowledge (geological information, well-log data, regional models, etc.). The formulation in Equation \eqref{youhua} incorporates geological priors $q_0$ as constraints, guiding the inversion process to select an optimal solution $\hat{q}$ that simultaneously satisfies both the observed seismic data $\lambda^*$ and aligns optimally with existing geological understanding $q_0$.

4. When engineers aim to design advanced composite materials with precise resonant frequencies $\lambda_*$ under specific vibration modes, this optimization framework proves indispensable. Starting from an initial density distribution $q_0$, the model performs fine-scale adjustments to the material properties. By solving this optimization problem, one can derive the optimal distribution of material density or elastic modulus $\hat{q}$ that not only satisfies the rigorous frequency requirements but also remains maximally faithful to the original design concept.

\section*{Acknowledgement}
\hskip\parindent
\small
We
declare that the authors are ranked in alphabetic order of their names and all of them have the same
contributions to this paper.

\end{document}